\documentclass[10pt]{amsart}
\usepackage{a4}
\usepackage{amssymb}
\usepackage{array}
\usepackage{booktabs}
\usepackage{multirow}
\usepackage{hhline}
\usepackage{graphicx,color}
\usepackage{bm}

\newtheorem{theorem}{Theorem}[section]

\newtheorem{proposition}[theorem]{Proposition}

\theoremstyle{definition}

\newtheorem{example}[theorem]{Example}

\theoremstyle{remark}
\newtheorem{remark}[theorem]{Remark}

\numberwithin{equation}{section}

\begin{document}
\title[Kronecker Product Graphs]{Kronecker Product Graphs and \\
Counting Walks in Restricted Lattices}



\author{Hun Hee Lee}
\address{Hun Hee Lee : Department of Mathematical Sciences, Seoul National University,
San56-1 Shinrim-dong Kwanak-gu, Seoul 151-747, Republic of Korea}
\email{hunheelee@snu.ac.kr}
\thanks{}

\author{Nobuaki Obata}
\address{Nobuaki Obata: Graduate School of Information Sciences, Tohoku University, Sendai, 980-8579 Japan}
\email{obata@math.is.tohoku.ac.jp}
\thanks{This work is supported by the Japan-Korea Basic
Scientific Cooperation Program ``Non-commutative Stochastic Analysis:
New Prospects of Quantum White Noise and Quantum Walk" (2015-2016)
and partially by JSPS Grant-in-Aid for Exploratory Research
No.~26610018.}


\date{}
\dedicatory{}

\commby{}

\begin{abstract}
Formulas are derived for counting walks in the Kronecker product of graphs,
and the associated spectral distributions are obtained by 
the Mellin convolution of probability distributions.
Two-dimensional restricted lattices admitting the Kronecker product structure
are listed, and their spectral distributions are calculated
in terms of elliptic integrals.
\end{abstract}

\maketitle


\bibliographystyle{amsplain}

\section{Introduction}

Counting walks in a graph is a basic and interesting problem.
Let $G=(V,E)$ be a locally finite graph with adjacency matrix $A$.
Then the matrix element $(A^m)_{xy}$ counts the number of $m$-step walks 
connecting $x$ and $y$.
For $x=y=o\in V$ this number is expressible in the integral:
\[
(A^m)_{oo}=\int_{\mathbb{R}} x^m \mu(dx),
\qquad m=0,1,2,\dots,
\]
where $\mu$ is a probability distribution on $\mathbb{R}=(-\infty,+\infty)$,
called the \textit{spectral distribution} of $A$ at a vertex $o$.
Thus the number of walks may be studied from 
an analytic or probabilistic point of view.
During the last fifteen years the quantum probability has been employed
for the asymptotic analysis of graph spectra 
as well as the study of product structures in connection with several
notions of independence,
see e.g., a monograph \cite{Hora-Obata}.

This paper focuses on 
the notion of \textit{Kronecker product} of graphs $G_1\times_K G_2$,
which is also called the \textit{direct product} 
and is one of the most basic graph products,
see the comprehensive monographs 
\cite{Hammack-Imrich-Klavzar2011}, \cite{Imrich-Klavzar2000}.
It is well known that the spectral distribution of
the Cartesian product of two graphs $G_1\times_C G_2$
is obtained by the usual convolution of probability distributions
defined by
\[
\int_{\mathbb{R}}h(x)\mu_1*\mu_2(dx)
=\int_{\mathbb{R}}\int_{\mathbb{R}}h(x+y)\mu_1(dx)\mu_2(dy),
\qquad h\in C_{\mathrm{bdd}}(\mathbb{R}).
\]
The convolution $\mu_1*\mu_2$ is known to be the distribution of 
the sum of two independent random variables $X_1+X_2$. 
Quantum probability allows us to discuss variations of 
independence of non-commutative variables.
The comb product of graphs is related to
the monotone convolution,
the star product to the Boolean convolution,
and the free product to the free convolution,
see e.g., \cite{Hora-Obata}, for further relevant results see
\cite{Accardi-Lenczewski-Salapata2007}.

The \textit{Mellin convolution}
in the original sense is the convolution product 
on the locally compact abelian group $(\mathbb{R}_{>0},\cdot)$ 
with the Haar measure $dx/x$ defined by
\begin{equation}\label{eq-Mellin-Convolution}
f\star g(x)
=\int_0^\infty f(y)g\Big(\frac{x}{y}\Big) \frac{dy}{y}
=\int_0^\infty f\Big(\frac{x}{y}\Big)g(y) \frac{dy}{y}
\end{equation}
for $f,g \in L^1((0,\infty),dx/x)$,
see e.g., \cite{Marichev}.
Extending the above definition naturally to 
symmetric probability distributions on $\mathbb{R}$,
we define the \textit{Mellin convolution} $\mu_1*_M\mu_2$ 
of two symmetric distributions $\mu_1$ and $\mu_2$ by 
\[
\int_{\mathbb{R}}h(x)\mu_1*_M\mu_2(dx)
=\int_{\mathbb{R}}\int_{\mathbb{R}}h(xy)\mu_1(dx)\mu_2(dy),
\qquad h\in C_{\mathrm{bdd}}(\mathbb{R}).
\]
Recall that a measure $\mu$ on $\mathbb{R}$ is called symmetric 
if  $\mu(-dx)=\mu(dx)$.
The Kronecker product of graphs becomes 
a new member of the corresponding list of ``product structures" 
of graphs and ``convolution products" 
of probability distributions on $\mathbb{R}$.

This paper is organized as follows.
In Section \ref{Sec:Counting walks in a graph} we assemble
basic notations and notions 
for counting walks in terms of the spectral distribution.
In Section \ref{Sec:Kronecker product of graphs} we introduce 
the concept of Kronecker product of graphs and 
show some elementary properties with illustrations.
The main result is stated in Theorem \ref{main theorem}.
In Section \ref{Sec:Subgraphs of 2-dimensional lattice as Kronecker products}
two-dimensional integer lattices restricted to certain domains
which admit the Kronecker product structure.
We derive formulas for counting walks and show that
the density functions of the spectral distributions
are expressible in terms of elliptic integrals.
Finally in Section \ref{sec:Examples in higher dimension} we
discuss towards higher dimensional extension,
where we find unexpectedly that
the restricted integer lattice $\{x\ge y\ge z\}$
and the mixed product $(\mathbb{Z}_+\times_K \mathbb{Z}_+)\times_C \mathbb{Z}_+$
are not isomorphic but have a common spectral distribution
at the origin $(0,0,0)$.

\section{Counting walks in a graph}
\label{Sec:Counting walks in a graph}

A graph $G=(V,E)$ is a pair,
where $V$ is a non-empty set and $E$ a subset of two-point subsets of $V$,
i.e., $E\subset\{\{x,y\}\,;\, x,y\in V, x\neq y\}$.
We deal with both finite and infinite graphs.
If $\{x,y\}\in E$, we say that $x$ and $y$ are adjacent and
write $x\sim y$.
The degree of $x\in V$ is defined to be the number of
vertices that are adjacent to $x$, and is denoted by $\deg x=\deg_G x$.
A graph under consideration in this paper is always assumed to be
locally finite, i.e., $\deg x<\infty$ for all vertices $x\in V$.

For $m=1,2,\dots$ an $m$-step walk from a vertex $x\in V$ to another $y\in V$ 
is an (ordered) sequence of vertices $x_0,x_1,\dots,x_m$ such that
\[
x=x_0\sim x_1\sim x_2\sim \dotsb\sim x_{m-1}\sim x_m=y.
\]
The number of such walks is interesting to study.
The adjacency matrix of a graph $G=(V,E)$ is a matrix
$A$ indexed by $V\times V$ whose entries are defined by
\[
(A)_{xy}=
\begin{cases}
1, &\text{if $x\sim y$}; \\
0, &\text{otherwise}.
\end{cases}
\]
By local finiteness the powers of $A$ are well-defined and
the matrix entry $(A^m)_{xy}$ counts the number of
$m$-step walks connecting $x$ and $y$.
It is convenient to introduce the Hilbert space $\ell^2(V)$ of 
$\mathbb{C}$-valued square-summable functions on $V$ with the inner product
\[
\langle f, g\rangle=\sum_{x\in V}\overline{f(x)}\,g(x),
\qquad
f,g\in \ell^2(V).
\]
Let $\{\delta_x\,;\,x\in V\}$ be the canonical orthonormal basis of
$\ell^2(V)$.
Then we have
\[
(A^m)_{xy}=\langle \delta_x,A^m\delta_y\rangle,
\quad m=0,1,2,\dots.
\]
We are particularly interested in counting the number of walks
from a vertex $o\in V$ to itself, which is denoted by 
\[
W_m(o;G)=(A^m)_{oo}\,,
\qquad m=0,1,2,\dots.
\]
We tacitly understand that $W_0(o;G)=1$.

\begin{theorem}\label{2thm:spectral distribution}
Let $G=(V,E)$ be a graph with a distinguished vertex $o\in V$.
Then there exists a probability distribution $\mu$ on $\mathbb{R}$ such that
\[
W_m(o;G)
=(A^m)_{oo}
=\langle \delta_o,A^m\delta_o\rangle
=M_m(\mu),
\qquad m=0,1,2,\dots,
\]
where 
\[
M_m(\mu)=\int_{\mathbb{R}} x^m \mu(dx)
\]
is the $m$-th moment of $\mu$.
\end{theorem}

The proof is by the Hamburger theorem, see e.g., \cite{Hora-Obata}.
The probability distribution in Theorem \ref{2thm:spectral distribution}
is called the \textit{spectral distribution} of $A$ in the vector state 
at $o\in V$.
The spectral distribution is not uniquely determined in general
due to the indeterminate moment problem,
however, it is unique if the degrees of vertices are uniformly bounded,
i.e., if $\sup\{\deg (x)\,;\,x\in V\}<\infty$.
If $W_{2m+1}(o;G)=(A^{2m+1})_{oo}=0$ for all $m=0,1,2,\dots$,
the spectral distribution may be assumed to be symmetric.

\section{Kronecker product of graphs}
\label{Sec:Kronecker product of graphs}

\subsection{Definition and elementary properties}

Let $G_1=(V_1,E_1)$ and $G_2=(V_2,E_2)$ be two 
(finite or infinite) graphs with adjacency matrices $A^{(1)}$ and $A^{(2)}$,
respectively.
Let $V=V_1\times V_2$ be the Cartesian product set and
define a matrix $A$ indexed by $V\times V$ by
\[
(A)_{(x,y),(x^\prime,y^\prime)}
=A^{(1)}_{xx^\prime} A^{(2)}_{yy^\prime},
\qquad
(x,y),(x^\prime,y^\prime)\in V.
\]
Since $A$ is a symmetric matrix whose 
diagonal entries are all zero
and off-diagonal ones take values in $\{0,1\}$,
there exists a graph $G$ on $V=V_1\times V_2$ whose adjacency matrix is $A$,
or equivalently, whose edge set is given by
\[
E=\{\{(x,y),(x^\prime,y^\prime)\}\,;\,
(A)_{(x,y),(x^\prime,y^\prime)}=1\}.
\]
The above graph $G$ is called the \textit{Kronecker product} of $G_1$ and $G_2$,
and is denoted by
\[
G=G_1\times_K G_2\,.
\]
In other words, 
the Kronecker product of $G_1=(V_1,E_1)$ and $G_2=(V_2,E_2)$
is a graph on $V_1\times V_2$ with adjacency relation
$(x,y)\sim_K(x^\prime,y^\prime)\Longleftrightarrow x\sim x^\prime$ 
and $y\sim y^\prime$.

\begin{remark}
The term \textit{Kronecker product} appears in 
\cite{Brouwer-Haemers2010} for instance,
while there are many synonyms.
The \textit{direct product} is another common term 
used in \cite{Godsil1993}, \cite{Hammack-Imrich-Klavzar2011},
\cite{Imrich-Klavzar2000} and so forth.
In this paper we prefer to the former 
in order to avoid confusion with some terms in quantum probability. 
\end{remark}

Through the canonical unitary isomorphism 
$\ell^2(V_1\times V_2)\cong \ell^2(V_1)\otimes \ell^2(V_2)$ 
given by $\delta_{(x,y)}\leftrightarrow \delta_x\otimes\delta_y$,
the adjacency matrix $A$ of $G_1\times_K G_2$ is written as
\begin{equation}\label{3eqn:def od A for Kronecker product}
A=A^{(1)}\otimes A^{(2)}.
\end{equation}
In fact, by definition we have
\[
(A)_{(x,y),(x^\prime,y^\prime)}
=\langle\delta_{(x,y)}, A\delta_{(x^\prime,y^\prime)}\rangle
=\langle\delta_x\otimes \delta_y, A(\delta_{x^\prime}\otimes\delta_{y^\prime})\rangle
\]
and
\[
A^{(1)}_{xx^\prime} A^{(2)}_{yy^\prime}
=\langle\delta_x, A^{(1)}\delta_{x^\prime}\rangle
 \langle\delta_y, A^{(2)}\delta_{y^\prime}\rangle
=\langle\delta_x\otimes \delta_y, 
 (A^{(1)}\otimes A^{(2)})(\delta_{x^\prime}\otimes\delta_{y^\prime})\rangle,
\]
from which \eqref{3eqn:def od A for Kronecker product} follows.

We collect some elementary properties, of which the proofs are 
straightforward.
For further relevant results, see   
the comprehensive monographs \cite{Hammack-Imrich-Klavzar2011},
\cite{Imrich-Klavzar2000}.

\begin{proposition}
For any graphs $G_1,G_2,G_3$ we have
\begin{gather*}
G_1\times_K G_2 \cong G_2\times_K G_1\,,
\\
(G_1\times_K G_2)\times_K G_3
\cong
G_1\times_K (G_2\times_K G_3).
\end{gather*}
\end{proposition}

\begin{proposition}
Let $G_1=(V_1,E_1)$ and $G_2=(V_2,E_2)$ be two connected graphs
with $|V_1|\ge2$ and $|V_2|\ge2$.
Then the Kronecker product $G_1\times_K G_2$ has at most two
connected components.
\end{proposition}

\begin{proposition}
Let $P_1$ be the graph consisting of a single vertex.
Then for any graph $G=(V,E)$
the Kronecker product $P_1\times_K G$ is a graph on $V$ with no edges,
i.e., an empty graph on $V$.
\end{proposition}

The \textit{Cartesian product}
of two graphs $G_1$ and $G_2$,
denoted by $G_1\times_C G_2$, 
is a graph on $V_1\times V_2$ with adjacency matrix defined by
\[
(A)_{(x,y),(x^\prime,y^\prime)}
=A^{(1)}_{xx^\prime}\delta_{yy^\prime}
 +\delta_{xx^\prime}A^{(2)}_{yy^\prime},
\]
or equivalently under the isomorphism
$\ell^2(V_1\times V_2)\cong \ell^2(V_1)\otimes \ell^2(V_2)$,
\[
A=A^{(1)}\otimes I^{(2)}+I^{(1)}\otimes A^{(2)},
\]
where $I^{(i)}$ is the identity matrix indexed by 
$V_i\times V_i$ for $i=1,2$. 

The distance-2 graph of $G_1\times_C G_2$ is a graph on $V_1\times V_2$
with adjacency relation:
\begin{align}
(x,y)\sim(x^\prime,y^\prime)
&\Longleftrightarrow 
\mathrm{dis}_{G_1\times_C G_2}((x,y),(x^\prime,y^\prime))=2 
\label{3eqn:in proof 3.4}\\
&\Longleftrightarrow 
\mathrm{dis}_{G_1}(x,x^\prime)+\mathrm{dis}_{G_2}(y,y^\prime)=2.
\nonumber
\end{align}
It is then easy to see that
the Kronecker product $G_1\times_K G_2$ is a subgraph of
the distance-2 graph of $G_1\times_C G_2$.
However, $G_1\times_K G_2$ is not necessarily
an induced subgraph of the distance-2 graph of $G_1\times_C G_2$.

\subsection{Counting walks}

The Kronecker product of graphs has a significant property
from the viewpoint of counting walks.

\begin{theorem}\label{thm:number of walks of Kronecker product}
Let $G_1\times_K G_2$ be the Kronecker product of
two graphs $G_1=(V_1,E_1)$ and $G_2=(V_2,E_2)$.
For $o_1\in V_1$ and $o_2\in V_2$ we have
\[
W_m((o_1,o_2);G_1\times_K G_2)=W_m(o_1;G_1)W_m(o_2;G_2),
\qquad m=0,1,2,\dots.
\]
\end{theorem}

\begin{proof}
Let $A^{(1)}$ and $A^{(2)}$ denote the adjacency matrices of $G_1$ and $G_2$,
respectively.
Let $A$ be the adjacency matrix of the Kronecker product $G_1\times_K G_2$. 
Using the natural isomorphism $\ell^2(V_1\times V_2)\cong
\ell^2(V_1)\otimes\ell^2(V_2)$ and 
$A=A^{(1)}\otimes A^{(2)}$ as in \eqref{3eqn:def od A for Kronecker product}
we calculate as follows:
\begin{align*}
W_m((o_1,o_2);G_1\times_K G_2)
&=\langle \delta_{(o_1,o_2)},A^m\delta_{(o_1,o_2)}\rangle \\
&=\langle \delta_{o_1}\otimes\delta_{o_2},
 (A^{(1)}\otimes A^{(2)})^m\delta_{o_1}\otimes\delta_{o_2}\rangle \\
&=\langle \delta_{o_1}, (A^{(1)})^m\delta_{o_1}\rangle
  \langle \delta_{o_2}, (A^{(2)})^m\delta_{o_2}\rangle \\
&=W_m(o_1;G_1)W_m(o_2;G_2),
\end{align*}
which completes the proof.
\end{proof}

\subsection{Mellin convolution of 
symmetric probability distribution on $\mathbb{R}$}

We focus on symmetric probability distributions $\mu$ on $\mathbb{R}$ 
having finite moments of all orders.
Since $M_{2m+1}(\mu)=0$ holds for all $m=0,1,2,\dots$,
we are mostly interested in the even moments.
For such probability distributions $\mu$ and $\nu$,
there exists a probability distribution, denoted by $\mu*_M\nu$,
uniquely specified by
\[
\int_{\mathbb{R}}h(x)\mu*_M\nu(dx)
=\int_{\mathbb{R}}\int_{\mathbb{R}}h(xy)\mu(dx)\nu(dy),
\qquad h\in C_{\mathrm{bdd}}(\mathbb{R}).
\]
We call $\mu*_M\nu$ the \textit{Mellin convolution}.
It is easily seen that $\mu*_M\nu$ is symmetric
and has finite moments of all orders.
In fact, 

\begin{proposition}\label{prop:moments Mellin convolution}
$M_m(\mu*_M \nu)=M_m(\mu)M_m(\nu)$ for all $m=0,1,2,\dots$.
\end{proposition}

Combining Theorems
\ref{2thm:spectral distribution}, 
\ref{thm:number of walks of Kronecker product}
and Proposition \ref{prop:moments Mellin convolution},
we come to the following fundamental result.

\begin{theorem}\label{main theorem}
For $i=1,2$ let $G_i=(V_i,E_i)$ be a graph with a distinguished vertex $o_i$.
Let $\mu_i$ be the spectral distribution of the adjacency matrix $A^{(i)}$ of
$G_i$ in the vector state at $o_i$.
Assume that $\mu_i$ is symmetric,
or equivalently that $W_{2m+1}(G_i, o_i)=0$
for all $m=0,1,2,\dots$ and $i=1,2$.
Then we have
\[
W_m((o_1,o_2);G_1\times_K G_2)
=M_m(\mu_1*_M\mu_2),
\qquad m=0,1,2,\dots.
\]
In other words, the spectral distribution of 
the Kronecker product $G_1\times_K G_2$ in the vector state at 
$(o_1,o_2)$ is the Mellin convolution of $\mu_1$ and $\mu_2$.
\end{theorem}

The Mellin convolution is originally introduced 
on the basis of the locally compact abelian group $\mathbb{R}_{>0}=(0,\infty)$, 
see Introduction. In this connection we should note the following

\begin{proposition}\label{Prop:density of Mellin convolution}
Let $f(x)$ and $g(x)$ be symmetric density functions on $\mathbb{R}$
and consider the probability distributions
$\mu(dx)=f(x)dx$ and $\nu(dx)=g(x)dx$.
Then $\mu*_M\nu$ admits a symmetric density function $2f\star g(x)$,
where $f\star g$ is the (original) Mellin convolution 
defined in \eqref{eq-Mellin-Convolution}.
\end{proposition}

\begin{proof}
By definition, for a symmetric function
$h\in C_{\mathrm{bdd}}(\mathbb{R})$ we have
\begin{align*}
\int_{\mathbb{R}}h(x)\mu*_M\nu(dx)
&=\int_{\mathbb{R}}\int_{\mathbb{R}}h(xy)\mu(dx)\nu(dy) \\
&=4\int_0^{\infty}\int_0^{\infty}h(xy)f(x)g(y)dxdy \\
&=4\int_0^{\infty} g(y)dy \int_0^{\infty}h(x)f\Big(\frac{x}{y}\Big)
 \frac{dx}{y} \\
&=2\int_{\mathbb{R}} h(x) dx
  \int_0^{\infty}f\Big(\frac{x}{y}\Big)g(y)\frac{dy}{y}\,.
\end{align*}
Hence, $2f\star g(x)$ is the density function of $\mu*_M\nu$.
\end{proof}

For the readers' convenience we make comparison with the Cartesian product.
The classical convolution of two probability distributions 
$\mu$ and $\nu$ is
a probability distribution, denoted by $\mu*\nu$,
uniquely specified by
\[
\int_{\mathbb{R}}h(x)\mu*\nu(dx)
=\int_{\mathbb{R}}\int_{\mathbb{R}}h(x+y)\mu(dx)\nu(dy),
\qquad h\in C_{\mathrm{bdd}}(\mathbb{R}).
\]
By applying the binomial expansion we get the following.

\begin{proposition}\label{prop:M_m(G_1XCG_2)}
For $i=1,2$ let $G_i=(V_i,E_i)$ be a graph with a distinguished vertex $o_i$.
Let $\mu_i$ be the spectral distribution of the adjacency matrix $A^{(i)}$ of
$G_i$ in the vector state at $o_i$.
Then we have
\begin{align*}
W_m((o_1,o_2);G_1\times_C G_2)
& =\sum_{k=0}^m \binom{m}{k} W_k(o_1;G_1)W_{m-k}(o_2;G_2)\\
& =M_m(\mu_1*\mu_2),
\qquad m=0,1,2,\dots,
\end{align*}
where $\mu_1*\mu_2$ is the (classical) convolution.
In other words, the spectral distribution of 
the Cartesian product $G_1\times_C G_2$ in the vector state at 
$(o_1,o_2)$ is the convolution of $\mu_1$ and $\mu_2$.
\end{proposition}

\section{Subgraphs of 2-dimensional lattice as Kronecker products}
\label{Sec:Subgraphs of 2-dimensional lattice as Kronecker products}

\subsection{The Kronecker product $\mathbb{Z}\times_K \mathbb{Z}$}\enspace

In order to avoid confusion we use the symbol $\mathbb{Z}^2$ just for the
Cartesian product set.
The Kronecker product $\mathbb{Z}\times_K \mathbb{Z}$
is by definition a graph on 
$\mathbb{Z}^2=\{(u,v)\,;\, u,v\in \mathbb{Z}\}$ with adjacency relation:
\begin{equation}\label{eqn:Mellin adjacency (1)}
(u,v)\sim_K(u^\prime, v^\prime)
\quad\Longleftrightarrow\quad
u^\prime=u\pm 1
\quad\text{and}\quad
v^\prime=v\pm 1.
\end{equation}
While, the so-called 2-dimensional integer lattice
is a graph on $\mathbb{Z}^2$ with adjacency relation:
\[
(x,y)\sim(x^\prime, y^\prime)
\quad\Longleftrightarrow\quad
\begin{cases}
x^\prime=x\pm 1, \\
y^\prime=y,
\end{cases}
\text{or}\quad
\begin{cases}
x^\prime=x, \\
y^\prime=y\pm 1.
\end{cases}
\]
We see immediately from definition that
$\mathbb{Z}\times_K \mathbb{Z}$ has two connected components,
each of which is isomorphic to
the 2-dimensional integer lattice $\mathbb{Z}\times_C \mathbb{Z}$.
Denoting by $(\mathbb{Z}\times_K \mathbb{Z})^o$ 
the connected component of $\mathbb{Z}\times_K \mathbb{Z}$
containing $o=(0,0)$, 
we claim the following

\begin{theorem}\label{thm:ZXMZ}
$(\mathbb{Z}\times_K \mathbb{Z})^o
\cong\mathbb{Z}\times_C \mathbb{Z}$,
where the isomorphism preserves the origin.
\end{theorem}

Here we prepare a general result.

\begin{proposition}\label{4prop:induced subgraph}
For $i=1,2$ let $G_i=(V_i,E_i)$ be a graph
and $H_i=(W_i,F_i)$ an induced subgraph of $G_i$.
Then $H_1\times_K H_2$ is an induced subgraph of $G_1\times_K G_2$.
\end{proposition}

\begin{proof}
By definition the vertex set of $H_1\times_K H_2$ is $W_1\times W_2$.
For two verices $(x,y), (x^\prime,y^\prime)\in W_1\times W_2$
we have $(x,y)\sim(x^\prime,y^\prime)$ 
in $H_1\times_K H_2$ if and only if
$x\sim x^\prime$ in $H_1$ and
$y\sim y^\prime$ in $H_2$ by definition.
Since $H_1$ and $H_2$ are respectively induced subgraphs of $G_1$ and $G_2$,
the last condition is equivalent to that
$x\sim x^\prime$ in $G_1$ and
$y\sim y^\prime$ in $G_2$,
hence to that
$(x,y)\sim (x^\prime,y^\prime)$ in $G_1\times_K G_2$.
Consequently,
$H_1\times_K H_2$ is an induced subgraph of
$G_1\times_K G_2$ spanned by $W_1\times W_2$.
\end{proof}

\subsection{Subgraphs of 2-dimensional integer lattice}

For a subset $D\subset \mathbb{Z}^2$ let
$L[D]$ denote the lattice restricted to $D$,
i.e., the induced subgraph of $\mathbb{Z}\times_C\mathbb{Z}$
spanned by the vertices in $D$.
We are particularly interested in 
restricted lattices which admit Kronecker product structure.
Theorem \ref{thm:ZXMZ} says that
$\mathbb{Z}\times_C \mathbb{Z}=L[\mathbb{Z}^2]$ itself
is isomorphic to the Kronecker product $(\mathbb{Z}\times_K \mathbb{Z})^o$.

\begin{theorem}\label{thm:ZXMP5 and Z_+XMZ}
For $n\ge2$ we have
\[
L\{(x,y)\in\mathbb{Z}^2\,;\, x\ge y \ge x-(n-1)\}
\cong (P_n\times_K \mathbb{Z})^o,
\]
where the right-hand side stands for
the connected component of $P_n\times_K \mathbb{Z}$
containing $o=(0,0)$,
$P_n$ being the path on $\{0,1,\dots,n-1\}$.
Similarly,
\[
L\{(x,y)\in\mathbb{Z}^2\,;\, x\ge y \}
\cong (\mathbb{Z}_+\times_K \mathbb{Z})^o.
\]
\end{theorem}

\begin{proof}
The path $P_n$ is naturally
regarded as an induced subgraph of $\mathbb{Z}$
spanned by $\{0,1,2,n-1\}$.
It then follows from Proposition \ref{4prop:induced subgraph} that
$P_n\times_K\mathbb{Z}$ is an induced subgraph 
of $\mathbb{Z}\times_K \mathbb{Z}$.
Therefore, $(P_n\times_K \mathbb{Z})^o$ is 
an induced subgraph of $(\mathbb{Z}\times_K \mathbb{Z})^o$.
Then, in view of Figure \ref{fig:restricted lattices},
we see that $(P_n\times_K \mathbb{Z})^o$ is 
isomorphic to the induced subgraph of $\mathbb{Z}\times_C \mathbb{Z}$ 
spanned by $D=\{(x,y)\,;\, x\ge y \ge x-(n-1)\}$.
The second assertion is proved similarly.
\end{proof}
\begin{figure}[hbt]
\begin{center} 
\includegraphics[width=150pt,keepaspectratio,clip]{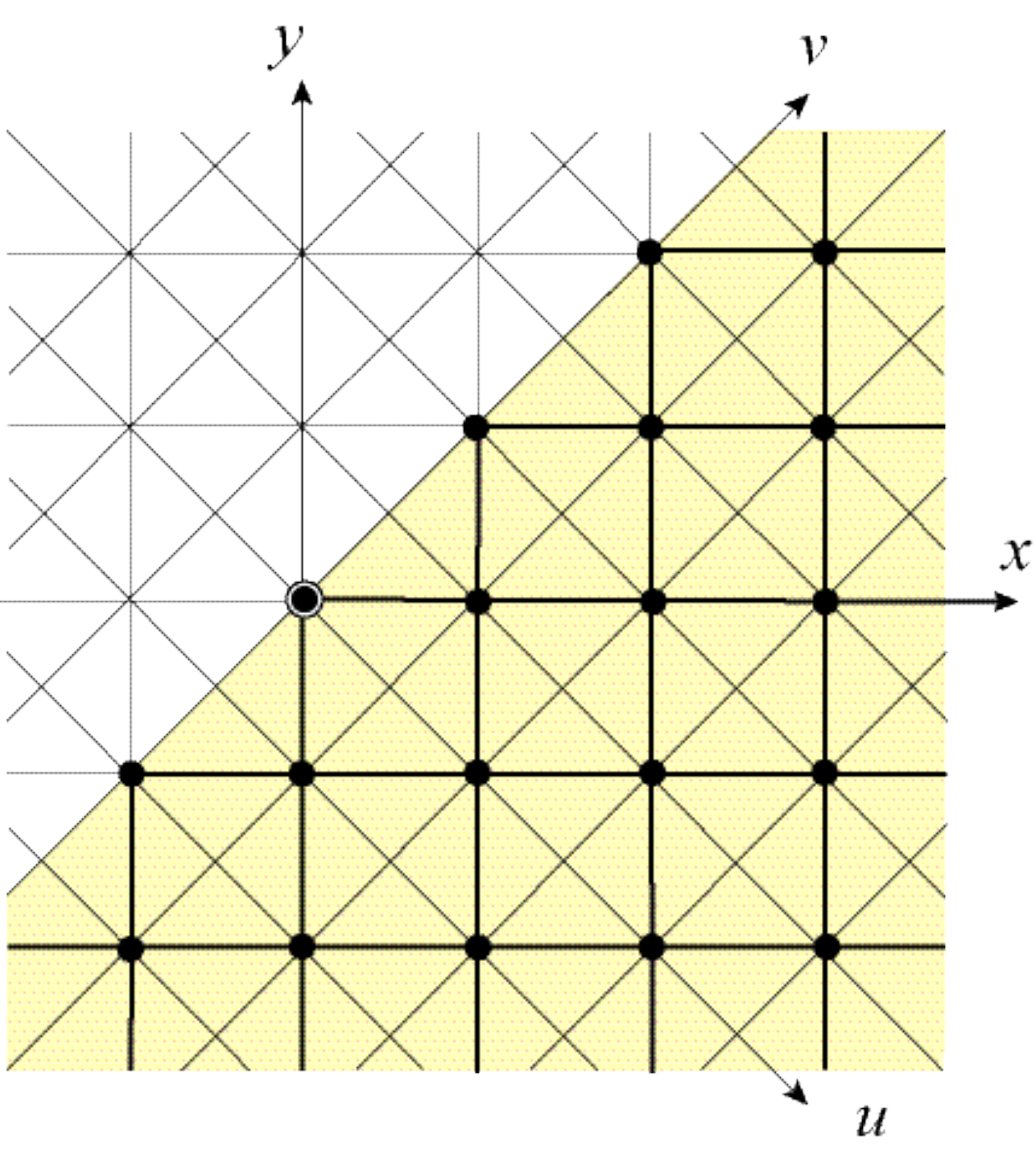}
\qquad
\includegraphics[width=150pt,keepaspectratio,clip]{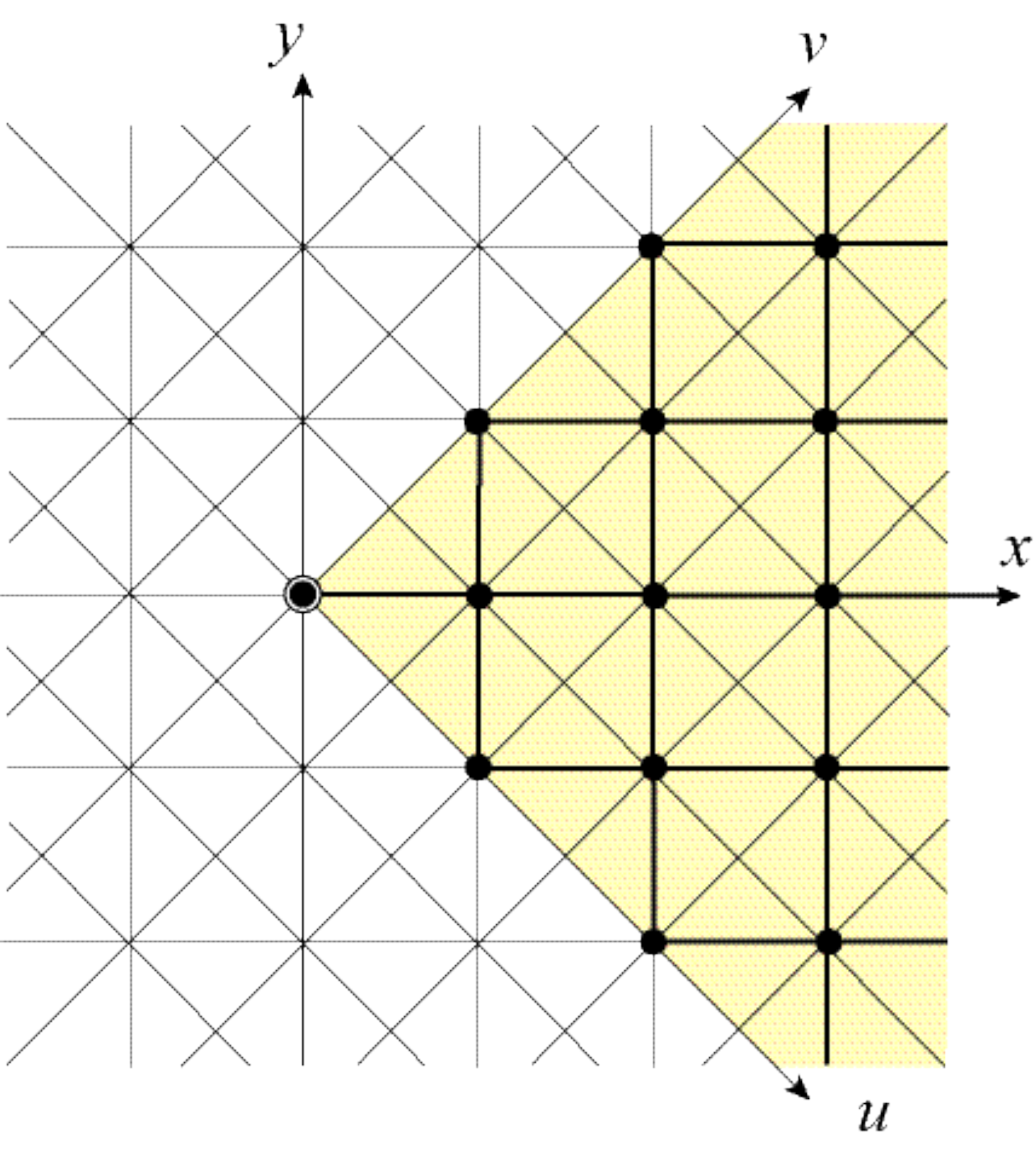}
\caption{$(\mathbb{Z}_+\times_K \mathbb{Z})^o\cong L(x\ge y)$
and $(\mathbb{Z}_+\times_K \mathbb{Z}_+)^o\cong L(-x\le y\le x)$}
\label{fig:restricted lattices} 
\end{center} 
\end{figure}

\begin{theorem}
\label{thm:PmXMPn}
For $k \ge2$ and $l\ge2$ we have
\[
L\left\{(x,y)\in\mathbb{Z}^2\,;\, 
\begin{array}{l}
0\le x+y\le k-1, \\
0\le x-y\le l-1
\end{array}
\right\}
\cong (P_k\times_K P_l)^o.
\]
Moreover,
\[
L\{(x,y)\in\mathbb{Z}^2\,;\, x\ge y\ge -x\}
\cong (\mathbb{Z}_+\times_K \mathbb{Z}_+)^o.
\]
\end{theorem}

The proof is similar as above, see also Figure \ref{fig:restricted lattices}.

\subsection{Counting walks}

The number of walks on one-dimensional integer lattice $\mathbb{Z}$
from the origin $0$ to itself is well known. 
We have
\begin{equation}\label{4eqn:walks in Z}
W_{2m}(0;\mathbb{Z})=\binom{2m}{m},
\quad
W_{2m+1}(0;\mathbb{Z})=0,
\quad m=0,1,2,\dots.
\end{equation}
A similar result for $\mathbb{Z}_+=\{0,1,2,\dots\}$ is also well known.
We have
\begin{equation}\label{4eqn:walks in Z+}
W_{2m}(0;\mathbb{Z}_+)=C_m=\frac{1}{m+1}\binom{2m}{m},
\quad
W_{2m+1}(0;\mathbb{Z}_+)=0,
\quad m=0,1,2,\dots,
\end{equation}
where $C_m$ is the renowned Catalan number.

We start with typical restricted lattices.

\begin{example}\label{thm:xgey}
(1) For $L=L\{(x,y)\in\mathbb{Z}^2\,;\, x\ge y \}$ we have
\[
W_{2m}((0,0);L)=C_m\binom{2m}{m}=\frac{1}{m+1}\binom{2m}{m}^2,
\quad m=0,1,2,\dots,
\]
and $W_{2m+1}((0,0);L)=0$.
Indeed, by Theorem \ref{thm:ZXMP5 and Z_+XMZ} we have
$L\cong (\mathbb{Z}_+\times_K \mathbb{Z})^o$,
where the origin $(0,0)$ in $L$
corresponds to $o=(0,0) \in\mathbb{Z}_+\times_K \mathbb{Z}$.
Hence
\[
W_{m}((0,0);L)
=W_m((0,0);\mathbb{Z}_+\times_K \mathbb{Z})
=W_m(0;\mathbb{Z}_+)W_m(0;\mathbb{Z}),
\]
where Theorem \ref{thm:number of walks of Kronecker product} is applied.
Then the result follows from 
\eqref{4eqn:walks in Z} and \eqref{4eqn:walks in Z+}.

(2)	For $L=L\{(x,y)\in\mathbb{Z}^2\,;\, x\ge y\ge -x\}$ we have 
\[
W_{2m}((0,0);L)=C_m^2=\frac{1}{(m+1)^2}\binom{2m}{m}^2,
\quad m=0,1,2,\dots,
\]
and $W_{2m+1}((0,0);L)=0$. 
Indeed, we get the result from Theorem \ref{thm:PmXMPn} 
along with a similar argument as in the previous example.

(3)	For $L[\mathbb{Z}^2]=\mathbb{Z}\times_C \mathbb{Z}$ we have
\begin{equation}\label{04eqn:counting walks in Z^2}
W_{2m}((0,0);\mathbb{Z}\times_C\mathbb{Z})
=\binom{2m}{m}^2,
\qquad m=0,1,2,\dots.
\end{equation}
Indeed, from Theorem \ref{thm:ZXMZ} we see that
$L[\mathbb{Z}^2]=\mathbb{Z}\times_C \mathbb{Z}
\cong (\mathbb{Z}\times_K \mathbb{Z})^o$.
Then we obtain
\begin{align*}
W_{2m}((0,0);\mathbb{Z}\times_C\mathbb{Z})
&=W_{2m}((0,0);\mathbb{Z}\times_K\mathbb{Z}) \\
&=W_{2m}(0;\mathbb{Z})W_{2m}(0;\mathbb{Z})
=\binom{2m}{m}^2,
\end{align*}
as desired.
Formula \eqref{04eqn:counting walks in Z^2} is derived in a different way. 
Applying Proposition \ref{prop:M_m(G_1XCG_2)} to
the Cartesian product $\mathbb{Z}\times_C \mathbb{Z}$, we obtain
\begin{align*}
W_{2m}((0,0);\mathbb{Z}\times_C\mathbb{Z})
&=\sum_{k=0}^m \binom{2m}{2k} W_{2k}(0;\mathbb{Z}) W_{2m-2k}(0;\mathbb{Z}) 
\label{4eqn:in ZXZ(2)} \\
&=\sum_{k=0}^m \binom{2m}{2k} \binom{2k}{k}
\binom{2m-2k}{m-k},
\nonumber
\end{align*}
where $W_{2m+1}(0;\mathbb{Z})=0$ is taken into account.
By comparing with \eqref{04eqn:counting walks in Z^2}
we get the following interesting relation:
\[
\sum_{k=0}^m \binom{2m}{2k} \binom{2k}{k}
\binom{2m-2k}{m-k}
=\binom{2m}{m}^2.
\]
Of course, one may calculate the left-hand side directly
by using the Vandermonde convolution formula for binomial coefficients
to get the right-hand side.
\end{example}

Finally we record the case where $D\subset \mathbb{Z}^2$ is bounded
in one or two directions,
see Theorems \ref{thm:ZXMP5 and Z_+XMZ} and \ref{thm:PmXMPn}.

\begin{example}
(1) For $L=L\{(x,y)\in\mathbb{Z}^2\,;\, x\ge y \ge x-(n-1)\}$ with
$n\ge2$ we have
\[
W_{2m}((0,0);L)
=W_{2m}(0;P_n)W_{2m}(0;\mathbb{Z})
=\binom{2m}{m} W_{2m}(0;P_n),
\qquad
m=0,1,2,\dots.
\]

(2)	For $L=L\{(x,y)\in\mathbb{Z}^2\,;\, 
0\le x+y\le k-1, \,\, 0\le x-y\le l-1\}$ with $k\ge2$ and $l\ge2$,
we have
\[
W_{2m}((0,0);L)=W_{2m}(0;P_k)W_{2m}(0;P_l),
\qquad
m=0,1,2,\dots.
\]
\end{example}

\begin{remark}\label{rem-moments-P_n}
A closed formula for $W_m(0;P_n)$ may be written down.
Set
\[
\lambda_k=2\cos\frac{k\pi}{n+1}\,,
\qquad k=1,2,\dots,n,
\]
which are, in fact, obtained from zeroes of the Chebyshev polynomials 
of the second kind.
We know that $\{\lambda_1,\dots,\lambda_n\}$ constitute 
the spectrum of $P_n$ (\cite[Section 1.4.4]{Brouwer-Haemers2010}).
Then there exist real constants $a_1,\dots,a_n$ such that
\begin{equation}\label{04eqn:in remark 4.12}
W_m(0;P_n)=\sum_{k=1}^n a_k \lambda_k^m,
\qquad m=0,1,2,\dots.
\end{equation}
Then, \eqref{04eqn:in remark 4.12} gives rise to 
a linear system $\bm{b}=\Lambda \bm{a}$.
For $m\le 2n$ we have
\[
W_m(0;P_n)=W_m(0;\mathbb{Z}_+)
=\begin{cases}
C_{m/2}, & \text{if $m$ is even}, \\
0, & \text{otherwise},
\end{cases}
\]
and the Vandermonde matrix $\Lambda$ is
easily inverted, we obtain $a_1,\dots,a_n$ uniquely from
$\bm{a}=\Lambda^{-1}\bm{b}$.
Here is a concrete example:
\[
W_{2m}(0;P_4)=\frac{5-\sqrt5}{10}\bigg(\frac{3+\sqrt5}{2}\bigg)^m
+\frac{5+\sqrt5}{10}\bigg(\frac{3-\sqrt5}{2}\bigg)^m
\]
for $m=0,1,2,\dots$,
and, of course, $W_{2m+1}(0;P_4)=0$.
\end{remark}

\subsection{Spectral distributions}
We will describe spectral distributions 
corresponding to graphs with Kronecker product structures. 
We begin with their building blocks, namely, 
spectral distributions associated to $\mathbb{Z}$, $\mathbb{Z}_+$ and $P_n$.

The \textit{arcsine distribution} with mean 0 and variance 2 
is defined by the density function:
\begin{equation}\label{eqn:arcsine}
\alpha(x)=\frac{1}{\pi\sqrt{4-x^2}}\,1_{(-2,2)}(x),
\qquad x\in\mathbb{R}.
\end{equation}
The \textit{semicircle distribution} with mean 0 and variance 1
is defined by the density function:
\begin{equation}\label{eqn:semicircle}
w(x)=\frac{1}{2\pi}\sqrt{4-x^2}\,1_{[-2,2]}(x),
\qquad x\in\mathbb{R}.
\end{equation}
By elementary calculus we have
\begin{align}
M_{2m}(\alpha)
&=\int_{\mathbb{R}} x^{2m} \alpha(x)\,dx
=\binom{2m}{m} = W_{2m}(0;\mathbb{Z}),
\label{eqn:W_{2m}(0;Z)}\\
M_{2m}(w)
&=\int_{\mathbb{R}} x^{2m} w(x)\,dx
=C_m=\frac{1}{m+1}\binom{2m}{m} 
=W_{2m}(0;\mathbb{Z}_+),
\label{eqn:W_{2m}(0;Z+)}
\end{align}
for $m=0,1,2,\dots$.
We see from Remark \ref{rem-moments-P_n} 
that the spectral distribution $\pi_n$ associated to $P_n$ is given by
\[
\pi_n = \sum_{k=1}^n a_k \delta_{\lambda_k},
\]
where $\delta_x$ is the Dirac measure on the point $x\in \mathbb{R}$.

Now we move to the 2-dimensional cases associated to Cartesian and Kronecker products.

\begin{example}
For the Cartesian product $\mathbb{Z}\times_C\mathbb{Z}$ 
we have
\begin{align*}
W_{m}((0,0);\mathbb{Z}\times_C\mathbb{Z})
&=\sum_{k=0}^m \binom{m}{k} W_{k}(0;\mathbb{Z})W_{m-k}(0;\mathbb{Z}) \\
&=\sum_{k=0}^m \binom{m}{k} M_{k}(\alpha) M_{m-k}(\alpha)
=M_{m}(\alpha*\alpha).
\end{align*}
While, for the Kronecker product we have
\begin{align*}
W_{m}((0,0);\mathbb{Z}\times_K\mathbb{Z})
&=W_{m}(0;\mathbb{Z})W_m(0;\mathbb{Z}) \\
&=M_{m}(\alpha) M_m(\alpha)
=M_{m}(\alpha *_M \alpha).
\end{align*}
Since $\mathbb{Z}\times_C\mathbb{Z}\cong
(\mathbb{Z}\times_K\mathbb{Z})^o$, we have
\begin{equation}\label{04eqn:M_m coincide}
M_{m}(\alpha*\alpha)=M_{m}(\alpha*_M\alpha),
\qquad m=0,1,2,\dots.
\end{equation}
Since $\alpha*\alpha$ (as well as $\alpha*_M \alpha$) has a compact support,
\eqref{04eqn:M_m coincide} is sufficient to claim
that $\alpha*\alpha=\alpha*_M\alpha$.
By similar argument we obtain
the spectral distributions for some restricted lattices.
The following table summarizes the results.
\smallskip
\begin{center}
\renewcommand{\arraystretch}{1.2}
\begin{tabular}{|c|c|c|}\hline
Domain $D$ & $W_{2m}(L[D],O)$ & spectral distribution \\ \hline
$\mathbb{Z}$ & $\binom{2m}{m}$ & $\alpha$ \\
$\mathbb{Z}_+$ & $C_m$ & $w$ \\ \hline
$\mathbb{Z}^2$ & $\binom{2m}{m}^2$  & $\alpha*\alpha=\alpha*_M\alpha$ \\
$\{x\ge y\}$ 
 & $C_m \binom{2m}{m}$  & $w*_M \alpha$ \\
$\{x\ge y\ge-x\}$
 & $C_m^2$  & $w*_M w$ \\
$\{x\ge0, \, y\ge0\}$
 & (A)  & $w* w$ \\
$\{x\ge y \ge x-(n-1)\}$
 & (B)  & $\pi_n*_M\alpha$ \\
$\left\{\begin{array}{l}
0\le x+y\le k-1, \\
0\le x-y\le l-1
\end{array}
\right\}$ & (C) & $\pi_k*_M \pi_l$
\\
\hline
\end{tabular}
\vspace*{10pt}
\end{center}

Concise formulas for (A)--(C) are not known, but we have 
\begin{gather*}
\mathrm{(A)}=\sum_{k=0}^m \binom{2m}{2k} C_kC_{m-k}, 
\qquad
\mathrm{(B)}=W_{2m}(0;P_n)\binom{2m}{m}, \\
\mathrm{(C)}=W_{2m}(0;P_k)W_{2m}(0;P_l).
\end{gather*}

\end{example}

\subsection{Calculating density functions}
In this section we investigate closed forms of density functions of 
the spectral distributions $\alpha*\alpha$, $w*_M\alpha$ and $w*_M w$.

\begin{example}
(1) It follows from Proposition \ref{Prop:density of Mellin convolution} that
the density function of $w*_M \alpha$ is given by $2w \star \alpha$. Since both $w(x)$ and $\alpha(x)$ are supported by 
the interval $[-2,2]$, we see easily that $w\star \alpha(x)=0$ for $x>4$. 
Then, in terms of the explicit forms \eqref{eqn:arcsine} 
and \eqref{eqn:semicircle}, we have:
\begin{align}\label{eqn:w star alpha}
w\star \alpha(x)
&=\int_0^{\infty} w(y)\alpha\Big(\frac{x}{y}\Big) \frac{dy}{y}\\
&=\frac{1}{2\pi^2} \int_{x/2}^2 \sqrt{4-y^2}\,
\frac{1}{\sqrt{4-(x/y)^2}}\, \frac{dy}{y} 
\nonumber\\
&=\frac{1}{2\pi^2} \int_{x/2}^2 
\sqrt{\frac{4-y^2}{4y^2-x^2}}\,dy,
\qquad 0\le x\le 4.
\nonumber
\end{align}
Here we need elliptic integrals and some relevant formulas \cite{Jeffrey}.
The complete elliptic integrals of the first and second kinds
are defined respectively by
\begin{align*}
K(k)&=\int_0^{\pi/2}\frac{d\theta}{\sqrt{1-k^2\sin^2 \theta}}
=\int_0^1 \frac{dx}{\sqrt{(1-x^2)(1-k^2x^2)}}, \\
E(k)&=\int_0^{\pi/2}\sqrt{1-k^2\sin^2 \theta}\, d\theta
=\int_0^1 \sqrt{\frac{1-k^2x^2}{1-x^2}}\,dx,
\end{align*}
where $k^2<1$.
Using the formula:
\[
\int_b^a \sqrt{\frac{a^2-t^2}{t^2-b^2}}\, dt
=a(K(k)-E(k)),
\quad
0<b<a,
\quad
k=\frac{\sqrt{a^2-b^2}}{a}\,,
\]
\eqref{eqn:w star alpha} becomes
\[
w\star \alpha(x)
=\frac{1}{2\pi^2}\{ K(\xi(x))-E(\xi(x)) \},
\]
where
\[
\xi(x)=\sqrt{1-\frac{x^2}{16}}\,.
\]
Consequently, the density function of $w*_M \alpha$ is given by
\[
\frac{1}{\pi^2}\{ K(\xi(x))-E(\xi(x)) \}1_{[-4,4]}(x),
\quad x\in\mathbb{R}.
\]
\begin{figure}[hbt]
\begin{center} 
\includegraphics[width=200pt,keepaspectratio,clip]{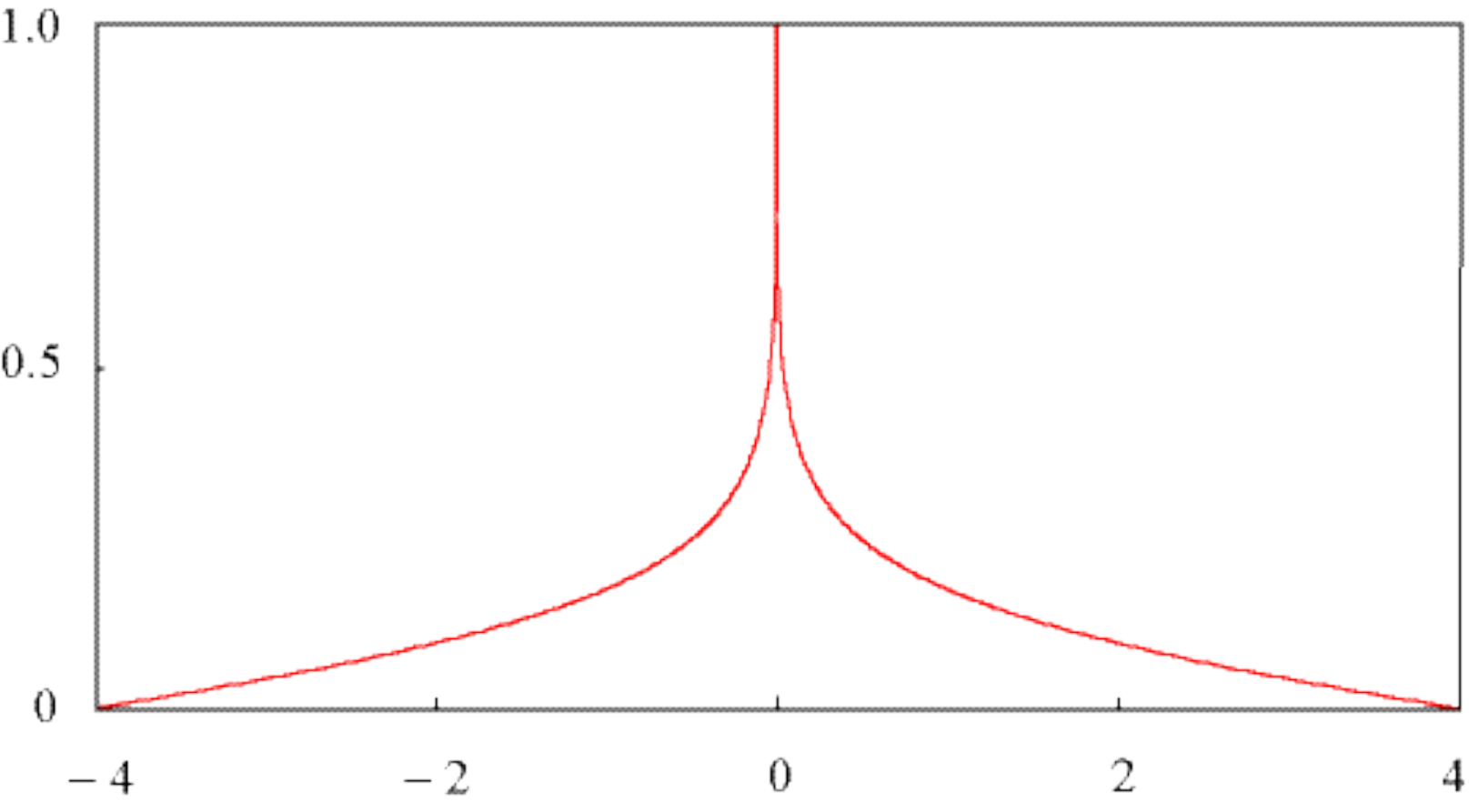}
\caption{The density function of $w*_M\alpha$}
\label{fig:wMa} 
\end{center} 
\end{figure}

(2)	Similarly, the density function of $\alpha*_M \alpha=\alpha*\alpha$ 
is given by
\[
\frac{1}{2\pi^2}\,
K(\xi(x))1_{[-4,4]}(x),
\qquad x\in\mathbb{R},
\]
and the density function of $w*_M w$ by
\[
\frac{2}{\pi^2}\left\{
\left(1+\frac{x^2}{16}\right)K(\xi(x))-2E(\xi(x))\right\}
1_{[-4,4]}(x),
\qquad x\in\mathbb{R}.
\]
\end{example}

\begin{figure}[hbt]
\begin{center} 
\includegraphics[width=200pt,keepaspectratio,clip]{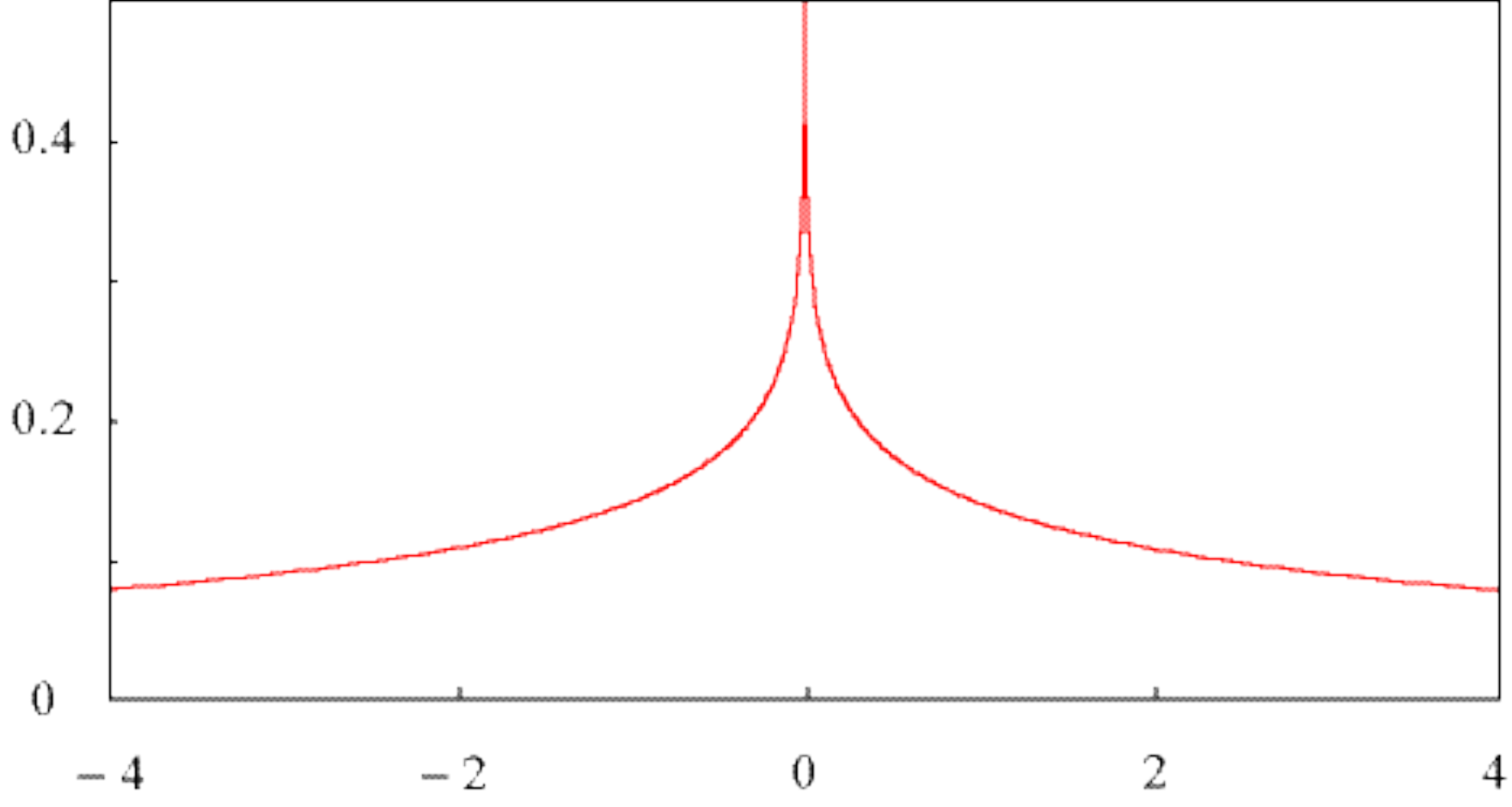}
\caption{The density function of $\alpha*_M\alpha$}
\label{fig:aMa} 
\end{center} 
\end{figure}
\begin{figure}[hbt]
\begin{center} 
\includegraphics[width=200pt,keepaspectratio,clip]{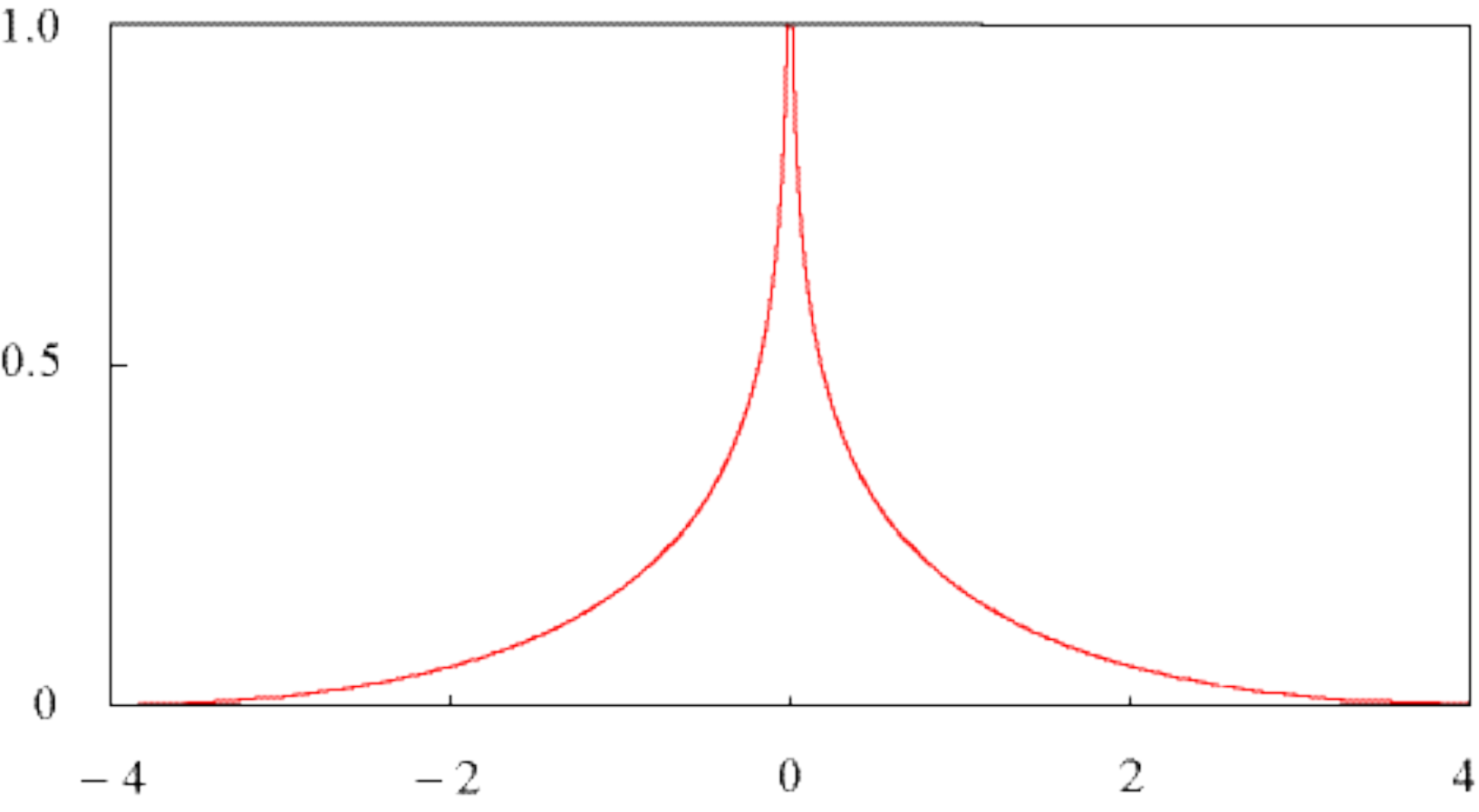}
\caption{The density function of $w*_Mw$}
\label{fig:wMw} 
\end{center} 
\end{figure}
%

\section{Examples in higher dimension}
\label{sec:Examples in higher dimension}

In this section we focus on some higher dimensional examples. 
We begin with all possible combinations of products on $\mathbb{Z}^3$, 
namely $\mathbb{Z}\times_K \mathbb{Z}\times_K\mathbb{Z}$ , 
$(\mathbb{Z}\times_K \mathbb{Z})\times_C\mathbb{Z}$, 
$(\mathbb{Z}\times_C \mathbb{Z})\times_K\mathbb{Z}$ 
and $\mathbb{Z}\times_C \mathbb{Z}\times_C\mathbb{Z}$.

\begin{example}\label{5ex:1}
(1) The Kronecker product $\mathbb{Z}\times_K \mathbb{Z}\times_K\mathbb{Z}$ 
has 4 connected components, which are mutually isomorphic.
We have
\[
W_{2m}((0,0,0);\mathbb{Z}\times_K \mathbb{Z}\times_K\mathbb{Z})
=\binom{2m}{m}^3,
\qquad m=0,1,2,\dots.
\]
The connected component containing $O(0,0,0)$, as is illustrated in
Figure \ref{fig:ZXMZXMZ}, is the body-centerd cubic lattice
or a kind of octahedral honeycomb.
For $(\mathbb{Z}\times_C \mathbb{Z})\times_K \mathbb{Z})^o$ we have 
\[
((\mathbb{Z}\times_C \mathbb{Z})\times_K \mathbb{Z})^o
\cong
((\mathbb{Z}\times_K \mathbb{Z})^o\times_K \mathbb{Z})^o
\cong 
(\mathbb{Z}\times_K \mathbb{Z}\times_K \mathbb{Z})^o.
\]
Hence counting walks in 
$(\mathbb{Z}\times_C \mathbb{Z})\times_K \mathbb{Z}$ is 
reduced to the previous one.
\begin{figure}[hbt]
\begin{center} 
\includegraphics[width=320pt,keepaspectratio,clip]{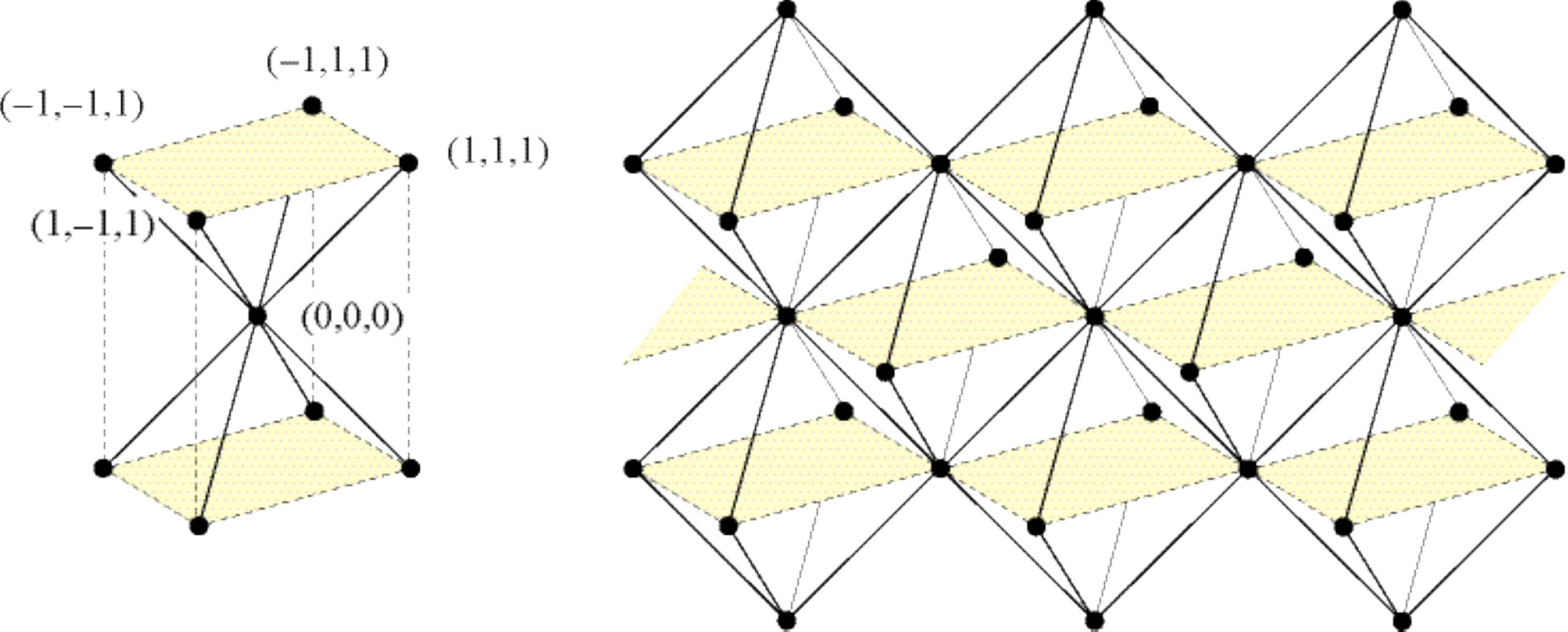}
\caption{$(\mathbb{Z}\times_K \mathbb{Z}\times_K\mathbb{Z})^o$}
\label{fig:ZXMZXMZ} 
\end{center} 
\end{figure}

(2) For other combinations of products of $\mathbb{Z}$ we see that
\[
((\mathbb{Z}\times_K \mathbb{Z})\times_C \mathbb{Z})^o
\cong (\mathbb{Z}\times_K \mathbb{Z})^o\times_C \mathbb{Z}
\cong (\mathbb{Z}\times_C \mathbb{Z})\times_C \mathbb{Z},
\]
which is the usual 3-dimensional integer lattice.
Hence
\begin{align*}
&W_{2m}((0,0,0);(\mathbb{Z}\times_C \mathbb{Z})\times_C\mathbb{Z}) 
=W_{2m}((0,0,0);(\mathbb{Z}\times_K \mathbb{Z})\times_C\mathbb{Z}) \\
&\qquad
=\sum_{k=0}^m \binom{2m}{2k}\binom{2k}{k}^2\binom{2m-2k}{m-k}
=\sum_{k=0}^m \frac{(2m)!(2k)!}{(m-k)!^2 k!^4}.
\end{align*}
Of course the above result is well known,
and our contribution here would be the derivation using the Kronecker product.
\end{example}

The last example is a very interesting case of products on $\mathbb{Z}^3_+$, which is related to a restricted lattice in $\mathbb{Z}^3$.

\begin{example}
The graph $(\mathbb{Z}_+\times_K \mathbb{Z}_+)\times_C \mathbb{Z}_+$
has two connected components
and we consider the connected component 
$((\mathbb{Z}_+\times_K \mathbb{Z}_+)\times_C \mathbb{Z}_+)^o$
containing $O=(0,0,0)$.
Then we have
\begin{align}
&W_{2m}((0,0,0);(\mathbb{Z}_+\times_K \mathbb{Z}_+)\times_C \mathbb{Z}_+)
\label{5eqn:ex 5.2} \\
&\qquad=\sum_{k=0}^m\binom{2m}{2k}
W_{2k}((0,0);\mathbb{Z}_+\times_K \mathbb{Z}_+)
W_{2m-2k}(0;\mathbb{Z}_+) 
\nonumber \\
&\qquad=\sum_{k=0}^m\binom{2m}{2k} C_k^2 C_{m-k} 
\nonumber \\
&\qquad=\sum_{k=0}^m\frac{(2m)!(2k)!}{(m-k)!(m-k+1)! k!^2(k+1)!^2}\,.
\nonumber
\end{align}
It is remarkable that
the last summation has been already obtained in \cite{Wimp-Zeilberger1989}
as the number of walks in the 3-dimensional
restricted lattice $L\{x\ge y\ge z\} = \{ (x,y,z) \in \mathbb{Z}^3 : x\ge y\ge z \}$, namely,
\[
W_{2m}((0,0,0);(\mathbb{Z}_+\times_K \mathbb{Z}_+)\times_C \mathbb{Z}_+)
=W_{2m}((0,0,0);L\{x\ge y\ge z\}),
\]
for all $m=0,1,2,\dots$.
It is, however, noted that 
$((\mathbb{Z}_+\times_K \mathbb{Z}_+)\times_C \mathbb{Z}_+)^o$
and $L\{x\ge y\ge z\}$ are not isomorphic.
For example, in the former graph
there is a unique vertex with degree 2 (that is, $O=(0,0,0)$),
while there are many vertices with degree 2 in the latter.
\end{example}

A similar phenomenon is observed also in the two-dimensional case.

\begin{example}
It follows by the usual reflection argument that
\[
W_{2m}(1;\mathbb{Z}_+)
=\binom{2m}{m}-\binom{2m}{m+2}
=C_{m+1}\,.
\]
On the other hand, it is known \cite{Guy-Krattenthaler-Sagan1992} that
\[
W_{2m}((0,0);\mathbb{Z}\times_C\mathbb{Z}_+)
=\binom{2m}{m}\binom{2m+2}{m}
 -\binom{2m+2}{m+1}\binom{2m}{m-1}
=C_mC_{m+1}\,.
\]
Therefore,
\[
W_{2m}((0,0);\mathbb{Z}\times_C\mathbb{Z}_+)
=W_{2m}((0,1);\mathbb{Z}_+\times_K\mathbb{Z}_+),
\]
though two graphs $\mathbb{Z}\times_C\mathbb{Z}_+$
and $\mathbb{Z}_+\times_K\mathbb{Z}_+$ are not isomorphic.
\end{example}

\end{document}